\documentclass{amsart}
\usepackage{amssymb}
\usepackage{amsfonts}

\newtheorem{theorem}{Theorem}
\theoremstyle{plain}

\newtheorem{corollary}{Corollary}

\newtheorem{lemma}{Lemma}

\numberwithin{equation}{section}
\input{tcilatex}

\begin{document}
\title[Simpson type inequalities]{Simpson type inequalities for functions
whose third derivatives in the absolute value are $s-$convex and $s-$concave}
\author{Merve Avci Ardic$^{\blacktriangledown \blacklozenge }$}
\address{$^{\blacktriangledown }$Ad\i yaman University, Faculty of Science
and Arts, Department of Mathematics, 02040, Ad\i yaman, Turkey.}
\email{merve.avci@atauni.edu.tr}
\thanks{$^{\blacklozenge }$corresponding author }
\author{M.E. Ozdemir$^{\bigstar }$}
\address{$^{\bigstar }$Atat\"{u}rk University, K.K. Education Faculty,
Department of Mathematics, 25240 Campus, Erzurum, Turkey}
\email{emos@atauni.edu.tr}

\begin{abstract}
In this paper, we established some new inequalities via $s-$convex and $s-$%
concave functions.
\end{abstract}

\maketitle

\section{introduction}

The following inequality is well known in the literature as Simpson's
inequality:%
\begin{eqnarray}
&&\left\vert \int_{a}^{b}f(x)dx-\frac{b-a}{3}\left[ \frac{f(a)+f(b)}{2}%
+2f\left( \frac{a+b}{2}\right) \right] \right\vert  \label{1.1} \\
&\leq &\frac{1}{2880}\left\Vert f^{(4)}\right\Vert _{\infty }\left(
b-a\right) ^{5},  \notag
\end{eqnarray}%
where the mapping $f:[a,b]\rightarrow 
\mathbb{R}
$ is assumed to be four times continuously differentiable on the interval
and $f^{(4)}$ to be bounded on $(a,b)$ , that is,%
\begin{equation*}
\left\Vert f^{(4)}\right\Vert _{\infty }=\sup_{t\in (a,b)}\left\vert
f^{(4)}(t)\right\vert <\infty .
\end{equation*}

For some results which generalize, improve and extend the inequality (\ref%
{1.1}), see the papers \cite{D}-\cite{U}.

In \cite{HM}, Hudzik and Maligranda considered among others the class of
functions which are $s-$convex in the second sense. This class is defined in
the following way: a function $f:%
\mathbb{R}
^{+}\rightarrow 
\mathbb{R}
$, where $%
\mathbb{R}
^{+}=[0,\infty ),$ is said to be s-convex in the second sense if%
\begin{equation*}
f(\alpha x+\beta y)\leq \alpha ^{s}f(x)+\beta ^{s}f(y)
\end{equation*}%
for all $x,y\in \lbrack 0,\infty ),$ $\alpha ,\beta \geq 0$ with $\alpha
+\beta =1$ and for some fixed $s\in (0,1].$ This class of $s-$convex
functions in the second sense is usually denoted by $K_{s}^{2}.$

It can be easily seen that for $s=1$, $s-$convexity reduces to ordinary
convexity of functions defined on $[0,1)$.

Some interesting and important inequalities for $s-$convex functions can be
found in \cite{HM}-\cite{AHO}.

In \cite{DF},Dragomir and Fitzpatrick proved a variant of Hadamard's
inequality which holds for s-convex functions in the second sense.

\begin{theorem}
\label{teo 1.0} Suppose that $f:[0,\infty )\rightarrow \lbrack 0,\infty )$
is an $s-$convex function in the second sense , where $s\in (0,1]$ and let $%
a,b\in \lbrack 0,\infty ),$ $a<b.$ If $f^{\prime }\in L^{1}[a,b],$ then the
following inequalities hold:%
\begin{equation}
2^{s-1}f\left( \frac{a+b}{2}\right) \leq \frac{1}{b-a}\int_{a}^{b}f(x)dx\leq 
\frac{f(a)+f(b)}{s+1}.  \label{1.2}
\end{equation}%
The constant $k=\frac{1}{s+1}$ is the best possible in the second inequality
in (\ref{1.2}). The above inequalities are sharp.
\end{theorem}

The main purpose of this paper is to establish some new inequalities for
functions whose third derivatives in the absolute value are $s-$convex and $%
s-$concave.

\section{inequalities for s-convex functions in the second sense}

To prove our new results we need the following lemma (see \cite{AS}).

\begin{lemma}
\label{lem 2.1} Let $f:I\rightarrow 
\mathbb{R}
$ be a function such that $f^{\prime \prime \prime \text{ }}$be absolutely
continuous on $I^{\circ }$, the interior of I. Assume that $a,b\in I^{\circ
},$ with $a<b$ and $f^{\prime \prime \prime \text{ }}\in L[a,b].$ Then, the
following equality holds,%
\begin{eqnarray*}
&&\int_{a}^{b}f(x)dx-\frac{b-a}{6}\left[ f(a)+4f\left( \frac{a+b}{2}\right)
+f(b)\right] \\
&=&\left( b-a\right) ^{4}\int_{0}^{1}p(t)f^{\prime \prime \prime
}(ta+(1-t)b)dt,
\end{eqnarray*}%
where%
\begin{equation*}
p(t)=\left\{ 
\begin{array}{c}
\frac{1}{6}t^{2}\left( t-\frac{1}{2}\right) ,\text{ \ \ \ }t\in \lbrack 0,%
\frac{1}{2}] \\ 
\\ 
\frac{1}{6}(t-1)^{2}\left( t-\frac{1}{2}\right) ,\text{ \ \ }t\in (\frac{1}{2%
},]\text{\ \ \ .\ \ \ \ }%
\end{array}%
\right.
\end{equation*}
\end{lemma}

\begin{theorem}
\label{teo 2.1} Let $f:I\subset \lbrack 0,\infty )\rightarrow 
\mathbb{R}
$ be a differentiable function on $I^{\circ }$ such that $f^{\prime \prime
\prime \text{ }}\in L[a,b],$ where $a,b\in I^{\circ }$ with $a<b.$ If $%
\left\vert f^{\prime \prime \prime \text{ }}\right\vert $ is $s-$convex in
the second sense on $[a,b]$ and for some fixed $s\in (0,1],$ then 
\begin{eqnarray*}
&&\left\vert \int_{a}^{b}f(x)dx-\frac{b-a}{6}\left[ f(a)+4f\left( \frac{a+b}{%
2}\right) +f(b)\right] \right\vert \\
&\leq &\frac{\left( b-a\right) ^{4}}{6}\left[ \frac{%
2^{-4-s}((1+s)(2+s)+34+2^{4+s}(-2+s)+11s+s^{2})}{(1+s)(2+s)(3+s)(4+s)}\right]
\\
&&\times \left[ \left\vert f^{\prime \prime \prime \text{ }}(a)\right\vert
+\left\vert f^{\prime \prime \prime \text{ }}(b)\right\vert \right] .
\end{eqnarray*}
\end{theorem}

\begin{proof}
From Lemma \ref{lem 2.1} and $s-$convexity of $\left\vert f^{\prime \prime
\prime \text{ }}\right\vert $ , we have%
\begin{eqnarray*}
&&\left\vert \int_{a}^{b}f(x)dx-\frac{b-a}{6}\left[ f(a)+4f\left( \frac{a+b}{%
2}\right) +f(b)\right] \right\vert \\
&\leq &\left( b-a\right) ^{4}\left\{ \int_{0}^{\frac{1}{2}}\left\vert \frac{1%
}{6}t^{2}\left( t-\frac{1}{2}\right) \right\vert \left\vert f^{\prime \prime
\prime }(ta+(1-t)b)\right\vert dt\right. \\
&&\left. +\int_{\frac{1}{2}}^{1}\left\vert \frac{1}{6}(t-1)^{2}\left( t-%
\frac{1}{2}\right) \right\vert \left\vert f^{\prime \prime \prime
}(ta+(1-t)b)\right\vert dt\right\} \\
&\leq &\frac{\left( b-a\right) ^{4}}{6}\left\{ \int_{0}^{\frac{1}{2}%
}t^{2}\left( \frac{1}{2}-t\right) \left( t^{s}\left\vert f^{\prime \prime
\prime \text{ }}(a)\right\vert +(1-t)^{s}\left\vert f^{\prime \prime \prime 
\text{ }}(b)\right\vert \right) dt\right. \\
&&\left. +\int_{\frac{1}{2}}^{1}(t-1)^{2}\left( t-\frac{1}{2}\right) \left(
t^{s}\left\vert f^{\prime \prime \prime \text{ }}(a)\right\vert
+(1-t)^{s}\left\vert f^{\prime \prime \prime \text{ }}(b)\right\vert \right)
dt\right\} \\
&=&\frac{\left( b-a\right) ^{4}}{6}\left[ \frac{%
2^{-4-s}((1+s)(2+s)+34+2^{4+s}(-2+s)+11s+s^{2})}{(1+s)(2+s)(3+s)(4+s)}\right]
\\
&&\times \left[ \left\vert f^{\prime \prime \prime \text{ }}(a)\right\vert
+\left\vert f^{\prime \prime \prime \text{ }}(b)\right\vert \right] ,
\end{eqnarray*}%
where we use the fact that 
\begin{equation*}
\int_{0}^{\frac{1}{2}}t^{2+s}\left( \frac{1}{2}-t\right) dt=\int_{\frac{1}{2}%
}^{1}(1-t)^{s+2}\left( t-\frac{1}{2}\right) =\frac{2^{-4-s}}{(3+s)(4+s)}
\end{equation*}%
and 
\begin{equation*}
\int_{0}^{\frac{1}{2}}t^{2}\left( \frac{1}{2}-t\right) (1-t)^{s}dt=\int_{%
\frac{1}{2}}^{1}(t-1)^{2}\left( t-\frac{1}{2}\right) t^{s}dt=\frac{%
2^{-4-s}\left( 34+2^{4+s}(-2+s)+11s+s^{2}\right) }{(1+s)(2+s)(3+s)(4+s)}.
\end{equation*}
\end{proof}

\begin{theorem}
\label{teo 2.2} Let $f:I\subset \lbrack 0,\infty )\rightarrow 
\mathbb{R}
$ be a differentiable function on $I^{\circ }$ such that $f^{\prime \prime
\prime \text{ }}\in L[a,b],$ where $a,b\in I^{\circ }$ with $a<b.$ If $%
\left\vert f^{\prime \prime \prime \text{ }}\right\vert ^{q}$ is $s-$convex
in the second sense on $[a,b]$ and for some fixed $s\in (0,1]$ and $q>1$
with $\frac{1}{p}+\frac{1}{q}=1,$ then the following inequality holds:%
\begin{eqnarray*}
&&\left\vert \int_{a}^{b}f(x)dx-\frac{b-a}{6}\left[ f(a)+4f\left( \frac{a+b}{%
2}\right) +f(b)\right] \right\vert \\
&\leq &\frac{\left( b-a\right) ^{4}}{48}\left( \frac{1}{2}\right) ^{\frac{1}{%
p}}\left( \frac{\Gamma (2p+1)\Gamma (p+1)}{\Gamma (3p+2)}\right) ^{\frac{1}{p%
}} \\
&&\times \left\{ \left[ \frac{1}{2^{s+1}(s+1)}\left\vert f^{\prime \prime
\prime \text{ }}(a)\right\vert ^{q}+\frac{2^{s+1}-1}{2^{s+1}(s+1)}\left\vert
f^{\prime \prime \prime \text{ }}(b)\right\vert ^{q}\right] ^{\frac{1}{q}%
}\right. \\
&&\left. +\left[ \frac{2^{s+1}-1}{2^{s+1}(s+1)}\left\vert f^{\prime \prime
\prime \text{ }}(a)\right\vert ^{q}+\frac{1}{2^{s+1}(s+1)}\left\vert
f^{\prime \prime \prime \text{ }}(b)\right\vert ^{q}\right] ^{\frac{1}{q}%
}\right\} .
\end{eqnarray*}
\end{theorem}

\begin{proof}
From Lemma \ref{lem 2.1}, using the $s-$convexity of $\left\vert f^{\prime
\prime \prime \text{ }}\right\vert ^{q}$ and the well-known H\"{o}lder's
inequality we have%
\begin{eqnarray*}
&&\left\vert \int_{a}^{b}f(x)dx-\frac{b-a}{6}\left[ f(a)+4f\left( \frac{a+b}{%
2}\right) +f(b)\right] \right\vert \\
&\leq &\frac{\left( b-a\right) ^{4}}{6}\left\{ \left( \int_{0}^{\frac{1}{2}%
}\left( t^{2}\left( \frac{1}{2}-t\right) \right) ^{p}dt\right) ^{\frac{1}{p}%
}\left( \int_{0}^{\frac{1}{2}}\left\vert f^{\prime \prime \prime
}(ta+(1-t)b)\right\vert ^{q}dt\right) ^{\frac{1}{q}}\right. \\
&&\left. +\left( \int_{\frac{1}{2}}^{1}\left( (t-1)^{2}\left( t-\frac{1}{2}%
\right) \right) ^{p}dt\right) ^{\frac{1}{p}}\left( \int_{\frac{1}{2}%
}^{1}\left\vert f^{\prime \prime \prime }(ta+(1-t)b)\right\vert
^{q}dt\right) ^{\frac{1}{q}}\right\} \\
&\leq &\frac{\left( b-a\right) ^{4}}{6}\left( \frac{\Gamma (2p+1)\Gamma (p+1)%
}{2^{3p+1}\Gamma (3p+2)}\right) ^{\frac{1}{p}} \\
&&\times \left\{ \left( \int_{0}^{\frac{1}{2}}\left[ t^{s}\left\vert
f^{\prime \prime \prime \text{ }}(a)\right\vert ^{q}+(1-t)^{s}\left\vert
f^{\prime \prime \prime \text{ }}(b)\right\vert ^{q}\right] dt\right) ^{%
\frac{1}{q}}\right. \\
&&\left. +\left( \int_{\frac{1}{2}}^{1}\left[ t^{s}\left\vert f^{\prime
\prime \prime \text{ }}(a)\right\vert ^{q}+(1-t)^{s}\left\vert f^{\prime
\prime \prime \text{ }}(b)\right\vert ^{q}\right] dt\right) ^{\frac{1}{q}%
}\right\} \\
&\leq &\frac{\left( b-a\right) ^{4}}{48}\left( \frac{1}{2}\right) ^{\frac{1}{%
p}}\left( \frac{\Gamma (2p+1)\Gamma (p+1)}{\Gamma (3p+2)}\right) ^{\frac{1}{p%
}} \\
&&\times \left\{ \left[ \frac{1}{2^{s+1}(s+1)}\left\vert f^{\prime \prime
\prime \text{ }}(a)\right\vert ^{q}+\frac{2^{s+1}-1}{2^{s+1}(s+1)}\left\vert
f^{\prime \prime \prime \text{ }}(b)\right\vert ^{q}\right] ^{\frac{1}{q}%
}\right. \\
&&\left. +\left[ \frac{2^{s+1}-1}{2^{s+1}(s+1)}\left\vert f^{\prime \prime
\prime \text{ }}(a)\right\vert ^{q}+\frac{1}{2^{s+1}(s+1)}\left\vert
f^{\prime \prime \prime \text{ }}(b)\right\vert ^{q}\right] ^{\frac{1}{q}%
}\right\}
\end{eqnarray*}%
where 
\begin{equation*}
\int_{0}^{\frac{1}{2}}\left( t^{2}\left( \frac{1}{2}-t\right) \right)
^{p}dt=\int_{\frac{1}{2}}^{1}\left( (t-1)^{2}\left( t-\frac{1}{2}\right)
\right) ^{p}dt=\frac{8^{-p}\Gamma (2p+1)\Gamma (p+1)}{2\Gamma (3p+2)}
\end{equation*}%
and $\Gamma $ is the Gamma function.
\end{proof}

\begin{corollary}
\label{co 2.1} If we choose $s=1$ in Theorem \ref{teo 2.2}, we have%
\begin{eqnarray*}
&&\left\vert \int_{a}^{b}f(x)dx-\frac{b-a}{6}\left[ f(a)+4f\left( \frac{a+b}{%
2}\right) +f(b)\right] \right\vert \\
&\leq &\frac{\left( b-a\right) ^{4}}{96}\left( \frac{1}{4}\right) ^{\frac{1}{%
q}}\left( \frac{\Gamma (2p+1)\Gamma (p+1)}{\Gamma (3p+2)}\right) ^{\frac{1}{p%
}} \\
&&\times \left\{ \left( \left\vert f^{\prime \prime \prime \text{ }%
}(a)\right\vert ^{q}+3\left\vert f^{\prime \prime \prime \text{ }%
}(b)\right\vert ^{q}\right) ^{\frac{1}{q}}+\left( 3\left\vert f^{\prime
\prime \prime \text{ }}(a)\right\vert ^{q}+\left\vert f^{\prime \prime
\prime \text{ }}(b)\right\vert ^{q}\right) ^{\frac{1}{q}}\right\} .
\end{eqnarray*}
\end{corollary}

\begin{theorem}
\label{teo 2.3} Suppose that all the assumptions of Theorem \ref{teo 2.2}
are satisfied. Then%
\begin{eqnarray*}
&&\left\vert \int_{a}^{b}f(x)dx-\frac{b-a}{6}\left[ f(a)+4f\left( \frac{a+b}{%
2}\right) +f(b)\right] \right\vert \\
&\leq &\frac{\left( b-a\right) ^{4}}{6}\left( \frac{1}{192}\right) ^{1-\frac{%
1}{q}} \\
&&\times \left\{ \left( \frac{2^{-4-s}}{(3+s)(4+s)}\left\vert f^{\prime
\prime \prime \text{ }}(a)\right\vert ^{q}+\frac{2^{-4-s}\left(
34+2^{4+s}(-2+s)+11s+s^{2}\right) }{(1+s)(2+s)(3+s)(4+s)}\left\vert
f^{\prime \prime \prime \text{ }}(b)\right\vert ^{q}\right) ^{\frac{1}{q}%
}\right. \\
&&\left. +\left( \frac{2^{-4-s}\left( 34+2^{4+s}(-2+s)+11s+s^{2}\right) }{%
(1+s)(2+s)(3+s)(4+s)}\left\vert f^{\prime \prime \prime \text{ }%
}(a)\right\vert ^{q}+\frac{2^{-4-s}}{(3+s)(4+s)}\left\vert f^{\prime \prime
\prime \text{ }}(b)\right\vert ^{q}\right) ^{\frac{1}{q}}\right\} .
\end{eqnarray*}
\end{theorem}

\begin{proof}
From Lemma \ref{lem 2.1} and using the well-known power-mean inequality we
have%
\begin{eqnarray*}
&&\left\vert \int_{a}^{b}f(x)dx-\frac{b-a}{6}\left[ f(a)+4f\left( \frac{a+b}{%
2}\right) +f(b)\right] \right\vert \\
&\leq &\frac{\left( b-a\right) ^{4}}{6}\left\{ \left( \int_{0}^{\frac{1}{2}%
}t^{2}\left( \frac{1}{2}-t\right) dt\right) ^{1-\frac{1}{q}}\left( \int_{0}^{%
\frac{1}{2}}t^{2}\left( \frac{1}{2}-t\right) \left\vert f^{\prime \prime
\prime }(ta+(1-t)b)\right\vert ^{q}dt\right) ^{\frac{1}{q}}\right. \\
&&\left. +\left( \int_{\frac{1}{2}}^{1}(t-1)^{2}\left( t-\frac{1}{2}\right)
dt\right) ^{1-\frac{1}{q}}\left( \int_{\frac{1}{2}}^{1}(t-1)^{2}\left( t-%
\frac{1}{2}\right) \left\vert f^{\prime \prime \prime
}(ta+(1-t)b)\right\vert ^{q}dt\right) ^{\frac{1}{q}}\right\} .
\end{eqnarray*}%
Since $\left\vert f^{\prime \prime \prime \text{ }}\right\vert ^{q}$ is $s-$%
convex, we have 
\begin{eqnarray*}
&&\int_{0}^{\frac{1}{2}}t^{2}\left( \frac{1}{2}-t\right) \left\vert
f^{\prime \prime \prime }(ta+(1-t)b)\right\vert ^{q}dt \\
&\leq &\int_{0}^{\frac{1}{2}}t^{2}\left( \frac{1}{2}-t\right) \left(
t^{s}\left\vert f^{\prime \prime \prime \text{ }}(a)\right\vert
+(1-t)^{s}\left\vert f^{\prime \prime \prime \text{ }}(b)\right\vert \right)
dt \\
&=&\frac{2^{-4-s}}{(3+s)(4+s)}\left\vert f^{\prime \prime \prime \text{ }%
}(a)\right\vert ^{q}+\frac{2^{-4-s}\left( 34+2^{4+s}(-2+s)+11s+s^{2}\right) 
}{(1+s)(2+s)(3+s)(4+s)}\left\vert f^{\prime \prime \prime \text{ }%
}(b)\right\vert ^{q}
\end{eqnarray*}%
and 
\begin{eqnarray*}
&&\int_{\frac{1}{2}}^{1}(t-1)^{2}\left( t-\frac{1}{2}\right) \left\vert
f^{\prime \prime \prime }(ta+(1-t)b)\right\vert ^{q}dt \\
&\leq &\int_{\frac{1}{2}}^{1}(t-1)^{2}\left( t-\frac{1}{2}\right) \left(
t^{s}\left\vert f^{\prime \prime \prime \text{ }}(a)\right\vert
+(1-t)^{s}\left\vert f^{\prime \prime \prime \text{ }}(b)\right\vert \right)
dt \\
&=&\frac{2^{-4-s}\left( 34+2^{4+s}(-2+s)+11s+s^{2}\right) }{%
(1+s)(2+s)(3+s)(4+s)}\left\vert f^{\prime \prime \prime \text{ }%
}(a)\right\vert ^{q}+\frac{2^{-4-s}}{(3+s)(4+s)}\left\vert f^{\prime \prime
\prime \text{ }}(b)\right\vert ^{q}.
\end{eqnarray*}%
Therefore we have 
\begin{eqnarray*}
&&\left\vert \int_{a}^{b}f(x)dx-\frac{b-a}{6}\left[ f(a)+4f\left( \frac{a+b}{%
2}\right) +f(b)\right] \right\vert \\
&\leq &\frac{\left( b-a\right) ^{4}}{6}\left( \frac{1}{192}\right) ^{1-\frac{%
1}{q}} \\
&&\times \left\{ \left( \frac{2^{-4-s}}{(3+s)(4+s)}\left\vert f^{\prime
\prime \prime \text{ }}(a)\right\vert ^{q}+\frac{2^{-4-s}\left(
34+2^{4+s}(-2+s)+11s+s^{2}\right) }{(1+s)(2+s)(3+s)(4+s)}\left\vert
f^{\prime \prime \prime \text{ }}(b)\right\vert ^{q}\right) ^{\frac{1}{q}%
}\right. \\
&&\left. +\left( \frac{2^{-4-s}\left( 34+2^{4+s}(-2+s)+11s+s^{2}\right) }{%
(1+s)(2+s)(3+s)(4+s)}\left\vert f^{\prime \prime \prime \text{ }%
}(a)\right\vert ^{q}+\frac{2^{-4-s}}{(3+s)(4+s)}\left\vert f^{\prime \prime
\prime \text{ }}(b)\right\vert ^{q}\right) ^{\frac{1}{q}}\right\} ,
\end{eqnarray*}%
which is the required result.
\end{proof}

\begin{corollary}
\label{co 2.2} If we choose $s=1$ in Theorem \ref{teo 2.3}, we have%
\begin{eqnarray*}
&&\left\vert \int_{a}^{b}f(x)dx-\frac{b-a}{6}\left[ f(a)+4f\left( \frac{a+b}{%
2}\right) +f(b)\right] \right\vert \\
&\leq &\frac{\left( b-a\right) ^{4}}{1152}\left\{ \left( \frac{3\left\vert
f^{\prime \prime \prime \text{ }}(a)\right\vert ^{q}+7\left\vert f^{\prime
\prime \prime \text{ }}(b)\right\vert ^{q}}{10}\right) ^{\frac{1}{q}}+\left( 
\frac{7\left\vert f^{\prime \prime \prime \text{ }}(a)\right\vert
^{q}+3\left\vert f^{\prime \prime \prime \text{ }}(b)\right\vert ^{q}}{10}%
\right) ^{\frac{1}{q}}\right\} .
\end{eqnarray*}%
The following result holds for $s-$concave functions.
\end{corollary}

\begin{theorem}
\label{teo 2.4} Let $f:I\subset \lbrack 0,\infty )\rightarrow 
\mathbb{R}
$ be a differentiable function on $I^{\circ }$ such that $f^{\prime \prime
\prime \text{ }}\in L[a,b],$ where $a,b\in I^{\circ }$ with $a<b.$ If $%
\left\vert f^{\prime \prime \prime \text{ }}\right\vert ^{q}$ is $s-$concave
on $[a,b]$ for $q>1$ with $\frac{1}{p}+\frac{1}{q}=1,$ then the following
inequality holds:%
\begin{eqnarray*}
&&\left\vert \int_{a}^{b}f(x)dx-\frac{b-a}{6}\left[ f(a)+4f\left( \frac{a+b}{%
2}\right) +f(b)\right] \right\vert \\
&\leq &\frac{\left( b-a\right) ^{4}}{48}\left( \frac{1}{2}\right) ^{\frac{1}{%
p}}\left( 2^{\frac{s-2}{q}}\right) \left( \frac{\Gamma (2p+1)\Gamma (p+1)}{%
\Gamma (3p+2)}\right) ^{\frac{1}{p}}\left\{ \left\vert f^{\prime \prime
\prime \text{ }}\left( \frac{a+3b}{4}\right) \right\vert +\left\vert
f^{\prime \prime \prime \text{ }}\left( \frac{3a+b}{4}\right) \right\vert
\right\} .
\end{eqnarray*}

\begin{proof}
From Lemma \ref{lem 2.1} and using the H\"{o}lder's inequality, we have%
\begin{eqnarray}
&&\left\vert \int_{a}^{b}f(x)dx-\frac{b-a}{6}\left[ f(a)+4f\left( \frac{a+b}{%
2}\right) +f(b)\right] \right\vert  \label{2.1} \\
&\leq &\frac{\left( b-a\right) ^{4}}{6}\left\{ \left( \int_{0}^{\frac{1}{2}%
}\left( t^{2}\left( \frac{1}{2}-t\right) \right) ^{p}dt\right) ^{\frac{1}{p}%
}\left( \int_{0}^{\frac{1}{2}}\left\vert f^{\prime \prime \prime
}(ta+(1-t)b)\right\vert ^{q}dt\right) ^{\frac{1}{q}}\right.  \notag \\
&&\left. +\left( \int_{\frac{1}{2}}^{1}\left( (t-1)^{2}\left( t-\frac{1}{2}%
\right) \right) ^{p}dt\right) ^{\frac{1}{p}}\left( \int_{\frac{1}{2}%
}^{1}\left\vert f^{\prime \prime \prime }(ta+(1-t)b)\right\vert
^{q}dt\right) ^{\frac{1}{q}}\right\} .  \notag
\end{eqnarray}%
Since $\left\vert f^{\prime \prime \prime \text{ }}\right\vert ^{q}$ is $s-$%
concave, using the inequality (\ref{1.2}), we have%
\begin{equation}
\int_{0}^{\frac{1}{2}}\left\vert f^{\prime \prime \prime
}(ta+(1-t)b)\right\vert ^{q}dt\leq 2^{s-2}\left\vert f^{\prime \prime \prime 
\text{ }}\left( \frac{a+3b}{4}\right) \right\vert ^{q}  \label{2.2}
\end{equation}%
and%
\begin{equation}
\int_{\frac{1}{2}}^{1}\left\vert f^{\prime \prime \prime
}(ta+(1-t)b)\right\vert ^{q}dt\leq 2^{s-2}\left\vert f^{\prime \prime \prime 
\text{ }}\left( \frac{3a+b}{4}\right) \right\vert ^{q}.  \label{2.3}
\end{equation}%
From (\ref{2.1})-(\ref{2.3}), we get%
\begin{eqnarray*}
&&\left\vert \int_{a}^{b}f(x)dx-\frac{b-a}{6}\left[ f(a)+4f\left( \frac{a+b}{%
2}\right) +f(b)\right] \right\vert \\
&\leq &\frac{\left( b-a\right) ^{4}}{48}\left( \frac{1}{2}\right) ^{\frac{1}{%
p}}\left( \frac{\Gamma (2p+1)\Gamma (p+1)}{\Gamma (3p+2)}\right) ^{\frac{1}{p%
}}2^{\frac{s-2}{q}}\left\{ \left\vert f^{\prime \prime \prime \text{ }%
}\left( \frac{a+3b}{4}\right) \right\vert +\left\vert f^{\prime \prime
\prime \text{ }}\left( \frac{3a+b}{4}\right) \right\vert \right\}
\end{eqnarray*}%
which completes the proof.
\end{proof}
\end{theorem}

\end{document}